\documentclass[12pt]{article}
\usepackage[centertags]{amsmath}
\usepackage{amsfonts}
\usepackage{amssymb}
\usepackage{amsthm}
\usepackage{newlfont}
\usepackage[all]{xy}
\usepackage[left=1in,top=1in,right=1in,nohead,foot=0.5in]{geometry}
 \newtheorem{thm}{Theorem}[section]
 \newtheorem{cor}[thm]{Corollary}
 \newtheorem{lem}[thm]{Lemma}
 \newtheorem{prop}[thm]{Proposition}
 \theoremstyle{definition}
 \newtheorem{defn}[thm]{Definition}
 \theoremstyle{remark}
 \newtheorem{rem}[thm]{Remark}
 \numberwithin{equation}{section}
 \newcommand{\Spec}{\operatorname{Spec}}
 
 \newcommand{\Res}{\operatorname{Res}}
 \newcommand{\Hom}{\operatorname{Hom}}
 \newcommand{\Rat}{\operatorname{Rat}}
 
 \newcommand{\SL}{\operatorname{SL}}
 \newcommand{\PGL}{\operatorname{PGL}}
 \newcommand{\Stab}{\operatorname{Stab}}

 \newcommand{\Sym}{\operatorname{Sym}}

\newlength{\defbaselineskip}
\begin{document}

\title{The Semistable Reduction Problem for the Space of Morphisms on $\mathbb{P}^{n}$}
\author{Alon Levy}

\maketitle

\section{Introduction and the Statement of the Problem}

\noindent The moduli spaces of dynamical systems on $\mathbb{P}^{n}$ are the spaces of morphisms, and more in general rational maps, defined by polynomials of
degree $d$; the case of interest is $d > 1$, in which case those rational maps are not automorphisms. For each $n$ and $d$, we write each rational map $\varphi$ as
$(\varphi_{0}:\ldots:\varphi_{n})$, so that the space is parametrized by the monomials of each $\varphi_{i}$ and is naturally isomorphic to a large projective
space, $\mathbb{P}^{N}$. By an elementary computation, $N = (n+1){n+d \choose d} - 1$. As we will not consider more than one of these moduli spaces at a time, there
is no ambiguity in writing just $N$, without explicit dependence on $n$ and $d$.

\medskip Within the space of rational maps, the space of morphisms is an affine open subvariety, denoted $\Hom_{d}^{n}$. The group $\PGL(n+1)$ acts on
$\mathbb{P}^{N}$ by conjugation, corresponding to coordinate change, i.e. $A$ maps $\varphi$ to $A\varphi A^{-1}$; this action preserves $\Hom_{d}^{n}$, since the
property of being a morphism is independent of coordinate change.

\medskip We study the quotient of the action using geometric invariant theory \cite{GIT}. To do this, we need to replace $\PGL(n+1)$ with $\SL(n+1)$, which projects
onto $\PGL(n+1)$ finite-to-one. Geometric invariant theory defines stable and semistable loci for the $\SL(n+1)$-action. To take the quotient, we need to remove the
unstable locus, defined as the complement of the semistable locus. The quotient of $\Hom_{d}^{n}$ by $\SL(n+1)$ is denoted $\mathrm{M}_{d}^{n}$, and parametrizes
morphisms on $\mathbb{P}^{n}$ up to coordinate change. The stable and semistable loci for the action of $\SL(n+1)$ on $\mathbb{P}^{N}$ are denoted by $\Hom_{d}^{n,
s}$ and $\Hom_{d}^{n, ss}$, and their quotients are denoted by $\mathrm{M}_{d}^{n, s}$ and $\mathrm{M}_{d}^{n, ss}$.

\medskip It is a fact that every regular map is in the stable locus. More precisely, we have the following prior results, due to \cite{Sil96}, \cite{PST}, and
\cite{Lev}:

\begin{thm}\label{containment}$\Hom_{d}^{n, s}$ and $\Hom_{d}^{n, ss}$ are open subvarieties of $\mathbb{P}^{N}$ such that $\Hom_{d}^{n} \subsetneq \Hom_{d}^{n, s}
\subseteq \Hom_{d}^{n, ss} \subsetneq \mathbb{P}^{N}$. The middle containment is an equality if and only if $n = 1$ and $d$ is even.\end{thm}

\begin{thm}\label{finite}The stabilizer group in $\PGL(n+1)$ of each element of $\Hom_{d}^{n}$ is finite and bounded in terms of $d$ and $n$.\end{thm}

\medskip $\mathrm{M}_{d}^{n, ss}$ is a proper variety, as it is the quotient of the largest semistable subspace of $\mathbb{P}^{N}$ for the action of $\SL(n+1)$.
We make the following simplifying,

\begin{defn}\label{ss}A rational map $\varphi \in \mathbb{P}^{N}$ is called semistable if it is in the semistable space $\Hom_{d}^{n, ss}$.\end{defn}

\medskip The semistable reduction theorem states the following, answering in the affirmative a conjecture for $\mathbb{P}^{1}$ in~\cite{STW}:

\begin{thm}\label{SSR1}If $C$ is a complete curve with $K(C)$ its function field, and if $\varphi_{K(C)}$ is a semistable rational map on $\mathbb{P}^{n}_{K(C)}$,
then there exists a curve $D$ mapping finite-to-one onto $C$ with a $\mathbb{P}^{n}$-bundle $\mathbf{P}(\mathcal{E})$ on $D$ with a self-map $\Phi$ such that,

\begin{enumerate}
  \item The restriction of $\Phi$ to the fiber of each $x \in D$, $\varphi_{x}$, is a semistable rational self-map.
  \item $\Phi$ is a semistable map over $K(D)$, and is equivalent to $\varphi_{K(D)}$ under coordinate change.
\end{enumerate}\end{thm}

\noindent This can be seen by using an alternative formulation. Semistable reduction can be thought of as extending a rational map defined over a field $K$ to a
rational map defined over a discrete valuation ring $R$ whose fraction field is $K$, in a way that is not too degenerate. The reason a discrete valuation ring
suffices is that once we know we can extend to a discrete valuation ring, we can extend to some larger integral domain.

\medskip We thus obtain the following equivalent formulation of semistable reduction:

\begin{thm}\label{SSR2}Let $G$ be a geometrically reductive group acting on a projective variety $X$ whose stable and semistable spaces are $X^{s}$ and $X^{ss}$
respectively. Let $R$ be a discrete valuation ring with fraction field $K$, and let $x_{K} \in X^{s}_{K}$. Then for some finite extension $K'$ of $K$, with $R'$ the
integral closure of $R$ in $K'$, $x_{K}$ has an integral model over $R'$ with semistable reduction modulo the maximal ideal. In other words, we can find some $A \in
G(\overline{K})$ such that $A\cdot x_{K}$ has semistable reduction. If $x_{K} \in X^{ss}_{K}$, then the same result is true, except that $x_{R'}$ could be an
integral model for some $x'_{K'}$ mapping to the same point of $X^{ss}//G$ such that $x'_{K'} \notin G\cdot x_{K}$.\end{thm}

\begin{proof}We follow the method used in~\cite{Zha}. Let $C$ be the Zariski closure of $x_{K}$ in $X^{ss}_{R}//G$, and reduce it modulo the maximal ideal to obtain
$x_{k}$, where $k$ is the residue field of $R$. Observe that $C$ is a one-dimensional subscheme of $X^{ss}_{\overline{k}}//G$ which is isomorphic to $\Spec R$, and
is as a result connected. Since $G$ is connected, the preimage $\pi^{-1}(C)$ is also connected: when $x_{K}$ is stable it follows from the fact that $\pi^{-1}(C)$
is the Zariski closure of $G\cdot x_{K}$ in $X^{ss}$, and even when it is not, $\pi^{-1}(C)$ is the union of connected orbits whose closures intersect. Since
further $\pi^{-1}(C)$ surjects onto $C$, we can find an integral one-dimensional subscheme mapping surjectively to $C$. This subscheme necessarily maps
finite-to-one onto $C$ by dimension counting, so it is isomorphic to some finite extension ring $R'$, giving us $K'$ as its fraction field.\end{proof}

\begin{rem}Theorem~\ref{SSR2} can also be proven in a much more explicit way, producing for each $\varphi_{K} \in \Hom_{d}^{n, ss}$ a sequence of $A$'s conjugating
it to a model with semistable reduction.\end{rem}

\medskip This leads to the natural question of which vector bundle classes can occur for each $C$, and more generally for each choice of $n$ and $d$. One
interesting subquestion is whether, for every $C$, we can choose the bundle to be trivial. Equivalently, it asks whether for each $C$ we can find a proper $D
\subseteq \Hom_{d}^{n, ss}$ that maps finite-to-one onto $C$. For most curves upstairs, the answer should be positive, by simple dimension counting: as demonstrated
in \cite{Sil96} and \cite{Lev}, the complement of $\Hom_{d}^{n, ss}$ has high codimension, equal to about half of $N$. However, it turns out that the answer is
sometimes negative, and in fact, for every $n$ and $d$ we can find a $C$ with only nontrivial bundle classes. More precisely:

\begin{thm}\label{bad1}For every $n$ and $d$, there exists a curve with no trivial bundle class satisfying semistable reduction.\end{thm}

\begin{rem}\label{bad1.1}An equivalent formulation for Theorem~\ref{bad1} is that for every $n$ and $d$ we can find a curve $C \subseteq \mathrm{M}_{d}^{n, ss}$
such that there does not exist a curve $D \subseteq \Hom_{d}^{n, ss}$ mapping onto $C$ under $\pi$.\end{rem}

\medskip Although most curves in $\Hom_{d}^{n, ss}$ can be completed, it does not imply we can find a nontrivial bundle on an open dense set of the Chow variety of
$\mathrm{M}_{d}^{n, ss}$. In fact, as we will see in section~\ref{goodcase}, there exist components of the Chow variety of $\mathrm{M}_{d}^{n, ss}$ where, at least
generically, a nontrivial bundle is required.

\medskip Our study of bundle classes will now split into two cases. In the case of curves satisfying semistable reduction with a trivial bundle, the reformulation
of Remark~\ref{bad1.1}, in its positive form, means that we can study $D$ directly as a curve in $\mathbb{P}^{N}$. We can bound the degree of the map from $D$ to
$C$ in terms of the stabilizer groups that occur on $D$. More precisely:

\begin{prop}\label{GIT}Let $X$ be a projective variety over an algebraically closed field with an action by a geometrically reductive linear algebraic group $G$.
Using the terminology of geometric invariant theory, let $D$ be a complete curve in the stable space $X^{s}$ whose quotient by $G$ is a complete curve $C$; say the
map from $D$ to $C$ has degree $m$. Suppose the stabilizer is generically finite, of size $h$, and either $D$ or $C$ is normal. Then there exists a finite subgroup
$S_{D} \subseteq G$, of order equal to $mh$, such that for all $x \in D$ and $g \in G$, $gx \in D$ iff $g \in S_{D}$.\end{prop}

\begin{cor}\label{GITc}With the same notation and conditions as in Proposition~\ref{GIT}, the map from $D$ to $C$ is ramified precisely at points $x \in D$ where
the stabilizer group is larger than $h$, and intersects $S_{D}$ in a larger subgroup than in the generic case.\end{cor}

\medskip If the genus of $C$ is $0$, then the only way the the map from $D$ to $C$ could have high degree is if it ramifies over many points; therefore,
Corollary~\ref{GITc} forces the degree to be small, at least as long as $C$ is contained in the stable locus.

\medskip In the case of curves that only satisfy semistable reduction with a nontrivial bundle, we do not have a description purely in terms of coordinates.
Instead, we will study which bundle classes can be attached to every curve $C$. The question of which bundles occur is an invariant of $C$; therefore, it is
essentially an invariant that we can use to study the scheme $\Hom(C, \mathrm{M}_{d}^{n, ss})$. In the sequel, we will study the scheme using the bundle class set
and height invariants.

\medskip For the study of which nontrivial bundle classes can occur, first observe that fixing a $D$ for which a bundle exists, we can apply the reformulation of
Theorem~\ref{SSR2} to obtain a unique extension of $\varphi$ locally. This can be done at every point, so it is true globally, so we have,

\begin{prop}\label{uni}Using the notation of Theorem~\ref{SSR1}, the bundle class $\mathbf{P}(\mathcal{E})$ depends only on $D$ and its trivialization $U_{i}$,
$U_{i} \hookrightarrow \Hom_{d}^{n, ss}$.\end{prop}

\medskip Note that the bundle class does not necessarily depend only on $D$, regarded as an abstract curve with a map to $C$. The reason is that a point of $D$ may
not be stable, which means it may correspond to one of several different orbits, whose closures intersect. However, there are only finitely many orbits
corresponding to each point, so the bundle class depends on $D$ up to a finite amount; if $C$ happens to be contained in the stable locus, then it depends only on
$D$.

\medskip Thus we can study which bundle classes occur for a given $C$. We will content ourselves with rational curves, for which there is a relatively easy
description of all projective bundles. Recall that every vector bundle over $\mathbb{P}^{1}$ splits as a direct sum of line bundles, and that the bundle
$\bigoplus_{i}\mathcal{O}(m_{i})$ is projectively equivalent to $\bigoplus_{i}\mathcal{O}(l + m_{i})$ for all $l \in \mathbb{Z}$. In other words, a
$\mathbb{P}^{n}$-bundle over $\mathbb{P}^{1}$ can be written as $\mathcal{O} \oplus \mathcal{O}(m_{1}) \oplus \ldots \oplus \mathcal{O}(m_{n})$; if the $m_{i}$'s
are in non-decreasing order, then the expression uniquely determines the bundle's class. We will show that,

\begin{prop}\label{multi}There exists a curve $C$ for which multiple non-isomorphic bundle classes can occur. In fact, suppose $C$ is isomorphic to $\mathbb{P}^{1}$,
and there exists $U \subseteq \Hom_{d}^{n, ss}$ mapping finite-to-one into $C$ such that $U$ is a projective curve minus a point. Then there are always infinitely
many possible classes: if the class of $U$ is thought of as splitting as $\mathbf{P}(\mathcal{E}) = \mathcal{O} \oplus \mathcal{O}(m_{1}) \oplus \ldots \oplus
\mathcal{O}(m_{n})$, where $m_{i} \in \mathbb{N}$, then for every integer $l$ the class $\mathcal{O} \oplus \mathcal{O}(lm_{1}) \oplus \ldots \oplus
\mathcal{O}(lm_{n})$ also occurs.\end{prop}

\medskip Proposition~\ref{multi} frustrated our initial attempt to obtain an easy classification of bundles based on curves. However, it raises multiple interesting
questions instead. First, the construction uses a rational $D$ mapping finite-to-one onto $C$, and going to higher $m$ involves raising the degree of the map $D \to
C$. It may turn out that bounding the degree bounds the bundle class; we conjecture that if we fix the degree of the map then we obtain only finitely many bundle
classes. Furthermore, in analogy with the consequences of Corollary~\ref{GITc}, we should conversely be able to bound the degree of the map in terms of $C$ and the
bundle class, at least for rational $C$.

\medskip Second, it is nontrivial to find the minimal $m_{i}$'s for which a bundle splitting as $\mathcal{O} \oplus \mathcal{O}(m_{1}) \oplus \ldots \oplus
\mathcal{O}(m_{n})$ would satisfy semistable reduction; the case of $n = 1$ could be stated particularly simply, as the question would be about the minimal $m$ for
which $\mathcal{O} \oplus \mathcal{O}(m)$ occurs.

\medskip In section~\ref{GITrecap}, we recap the basics of geometric invariant theory, which we will use in the proof of Theorem~\ref{bad1}. In sections~\ref{exbad}
and~\ref{pfbad} we will illustrate Theorem~\ref{bad1}: in section~\ref{exbad} we will give some examples and compute the bundle classes that occur, proving
Proposition~\ref{multi} on the way, while in section~\ref{pfbad} we will prove the theorem. In section~\ref{goodcase} we will focus on the trivial bundle case,
proving Proposition~\ref{GIT} and defining the height function, which will impose constraints on which curves admit a trivial bundle; this will allow us to obtain a
large family of curves $C$ in $\mathrm{M}_{2}^{ss}$ with no trivial bundle.

\section{A Description of The Stable and Semistable Spaces}\label{GITrecap}

\noindent Recall from geometric invariant theory that a when a geometrically reductive linear algebraic group $G$ has a linear action on a projectivized vector
space $\mathbb{P}(V)$, we have,

\begin{defn}\label{sss}A point $x \in V$ is called semistable (resp. stable) if any of the following equivalent conditions hold:

\begin{enumerate}
  \item There exists a $G$-invariant homogeneous section $s$ such that $s(x) \neq 0$ (resp. same condition, and the action of $G$ on $x$ is closed).
  \item The closure of $G\cdot x$ does not contain $0$ (resp. $G\cdot x$ is closed).
  \item Every one-parameter subgroup $T$ acts on $x$ with both nonnegative and nonpositive weights (resp. negative and positive weights).
\end{enumerate}\end{defn}

\begin{rem}The last condition in the definition is equivalent to having nonpositive (resp. negative) weights. This is because if we can find a subgroup acting with
only negative weights, then we can take its inverse and obtain only positive weights.\end{rem}

\medskip Observe that for every nonzero scalar $k$, $x$ is stable (resp. semistable) iff $kx$ is. So the same definitions of stability and semistability hold for
points of $\mathbb{P}(V)$. The definitions also descend to every $G$-invariant projective variety $X \subseteq \mathbb{P}(V)$; in fact, in~\cite{GIT} they are
defined for $X$ in terms of a $G$-equivariant line bundle $L$. When $L$ is ample, as in the case of the space under discussion in this paper, this reduces to the
above definition.

\medskip The importance of stability is captured in the following prior results:

\begin{prop}The space of all stable points, $X^{s}$, and the space of all semistable points, $X^{ss}$, are both open and $G$-invariant.\end{prop}

\begin{thm}There exists a quotient $Y = X^{ss}//G$, called a good categorical quotient, with a natural map $\pi:X \to Y$, satisfying the following properties:

\begin{enumerate}
  \item $\pi$ is a $G$-equivariant map, where $G$ acts on $Y$ trivially.
  \item Every $G$-equivariant map $X \to Z$, where $G$ acts on $Z$ trivially, factors through $\pi$.
  \item $\pi$ is an open submersion.
  \item $\pi(x_{1}) = \pi(x_{2})$ iff the closures of $G\cdot x_{1}$ and $G\cdot x_{2}$ intersect.
  \item For every open $U \subseteq Y$, $\mathcal{O}_{U} = \mathcal{O}(\pi^{-1}(U))^{G}$.
\end{enumerate}

\noindent In addition, $Y$ is proper.\end{thm}

\begin{thm}There exists a quotient $Z = X^{s}//G$, called a good geometric quotient, with a natural map $\pi:X \to Z$ satisfying all enumerated conditions of a good
categorial quotient, as well as the following:

\begin{enumerate}
  \item $\pi(x_{1}) = \pi(x_{2})$ iff $G\cdot x_{1} = G\cdot x_{2}$.
  \item $Z$ is naturally an open subset of $X^{ss}//G$.
\end{enumerate}\end{thm}

\begin{thm}On $X^{s}$, the dimension of the stabilizer group $\Stab_{G}(x)$ is constant.\end{thm}

\medskip Returning to our case of self-maps of $\mathbb{P}^{n}$, we write the stable and semistable spaces for the conjugation action as $\Hom_{d}^{n, s}$ and
$\Hom_{d}^{n, ss}$. This involves a fair amount of abuse of notation, since those two spaces are open subvarieties of $\mathbb{P}^{N}$ and in fact properly contain
$\Hom_{d}^{n}$, which consists only of regular maps.

\medskip In \cite{Lev} we proved the fact that $\Hom_{d}^{n} \subsetneq \Hom_{d}^{n, s}$ by describing $\Hom_{d}^{n, s}$ and $\Hom_{d}^{n, ss}$ more or less
explicitly. We will recap the results, which are very technical but help us answer the question of when we can obtain a trivial bundle class in the semistable
reduction problem and when we cannot.

\medskip We use the Hilbert-Mumford criterion, which is the last condition in Definition~\ref{sss}. In more explicit terms, the criterion for semistability (resp.
stability) states that for every one-parameter subgroup $T \leq \SL(n+1)$, the action of $T$ on $\varphi$ can be diagonalized with eigenvalues $t^{a_{I}}$ and at
least one $a_{I}$ is nonpositive (resp. negative). Now, assume by conjugation that this one-parameter subgroup is in fact diagonal, with diagonal entries
$t^{a_{0}}, \ldots, t^{a_{n}}$, and that $a_{0} \geq \ldots \geq a_{n}$; we may also assume that the $a_{i}$'s are coprime, as dividing throughout by a common
factor would not change the underlying group. Note also that $a_{0} + \ldots + a_{n} = 0$. Our task is made easy by the fact that our standard coordinates for
$\mathbb{A}^{N+1}$ are the monomials, on which $T$ already acts diagonally. Throughout this analysis, we fix $\mathbf{a} = (a_{0}, \ldots, a_{n})$, and similarly
for $\mathbf{x}$ and $\mathbf{d}$.

\medskip Now, $T$ acts on the $x_{0}^{d_{0}}\ldots x_{n}^{d_{n}}$ monomial of the $i$th polynomial, $\varphi_{i}$, with weight $a_{i} - \mathbf{a}\cdot\mathbf{d}.$
A map $\varphi \in \mathbb{P}^{N}$ is unstable (resp. not stable) iff, after conjugation, there exists a choice of $a_{i}$'s such that whenever the
$\mathbf{x^{d}}$-coefficient of $\varphi_{i}$ satisfies $\mathbf{a}\cdot\mathbf{d} \leq a_{i}$ (resp. $<$), it is equal to zero.

\begin{rem}While in principle there are infinitely many possible $T$'s, parametrized by a hyperplane in $\mathbb{P}^{n}(\mathbb{Q})$, in practice there are up to
conjugation only finitely many. This is because each diagonal $T$ imposes conditions of the form ``the $\mathbf{x^{d}}$-coefficient of $\varphi_{i}$ is zero,'' and
there are only finitely many such conditions. Thus the stable and semistable spaces are indeed open in $\mathbb{P}^{N}$.\end{rem}

\begin{rem}The conjugation conditions we have chosen for $T$ are such that the conditions they impose for $\varphi$ to be unstable (or merely not stable) are the most
stringent on $\varphi_{n}$ and least stringent on $\varphi_{0}$, and are the most stringent on monomials with high $x_{0}$-degrees and least stringent on monomials
with high $x_{n}$-degrees.\end{rem}

\medskip If $n = 1$, we have a simpler description, due to Silverman \cite{Sil96}:

\begin{thm}$\varphi \in \mathbb{P}^{N}$ is unstable (resp. not stable) iff it is equivalent under coordinate change to a map $(a_{0}x^{d} + \ldots +
a_{d}y^{d})/(b_{0}x^{d} + \ldots + b_{d}y^{d})$, such that:

\begin{enumerate}
  \item $a_{i} = 0$ for all $i \leq (d-1)/2$ (resp. $<$).
  \item $b_{i} = 0$ for all $i \leq (d+1)/2$ (resp. $<$).
\end{enumerate}\end{thm}

\medskip The description for $n = 1$ can be thought of as giving a dynamical criterion for stability and semistability. A point $\varphi \in \mathbb{P}^{N}$ is
unstable if there exists a point $x \in \mathbb{P}^{1}$ where $\varphi$ has a bad point of degree more than $(d+1)/2$, or $\varphi$ has a bad point of degree more
than $(d-1)/2$ where it in addition has a fixed point. Following Rahul Pandharipande's unpublished reinterpretation of \cite{Sil96}, we define ``bad point'' as a
vertical component of the graph $\Gamma_{\varphi} \subseteq \mathbb{P}^{1} \times \mathbb{P}^{1}$, and ``fixed point'' as a fixed point of the unique non-vertical
component of $\Gamma_{\varphi}$. When $n = 1, d = 2$, this condition reduces to having a fixed point at a bad point, or alternatively a repeated bad point.

\medskip The conditions for higher $n$ are not as geometric. However, if we interpret fixed points liberally enough, there are still strong parallels with the $n =
1$ case. One can show that the unstable space for $n = 2$ and $d = 2$ consists of two irreducible components, which roughly generalize the $n = 1, d = 2$ condition
of having a fixed point at a bad point; in this case, one needs to define a limit of the value of $\varphi(x)$ as $x$ approaches the bad point, though this limit
can be defined purely in terms of degrees of polynomials, without needing to resort to a specific metric on the base field.

\section{Examples of Nontrivial Bundles}\label{exbad}

\noindent The space $\Rat_{2} = \Hom_{2}^{1}$ and its quotient $\mathrm{M}_{2}$ have been analyzed with more success than the larger spaces, yielding the following
prior structure result \cite{D1C} \cite{Sil96}:

\begin{thm}\label{Sil}$\mathrm{M}_{2} = \mathbb{A}^{2}$; $\mathrm{M}_{2}^{s} = \mathrm{M}_{2}^{ss} = \mathbb{P}^{2}$. The first two elementary symmetric polynomials
in the multipliers of the fixed points realize both isomorphisms.\end{thm}

\medskip Recall that within $\mathbb{P}^{N} = \mathbb{P}^{5}$, a map $(a_{0}x^{2} + a_{1}xy + a_{2}y^{2})/(b_{0}x^{2} + b_{1}xy + b_{2}y^{2})$ is unstable iff it is
in the closure of the $\PGL(2)$-orbit of the subvariety $a_{0} = b_{0} = b_{1} = 0$. In other words, it is unstable iff there the map is degenerate and has a double
bad point, or a fixed point at a bad point.

\begin{defn}\label{poly}A map on $\mathbb{P}^{1}$ is a polynomial iff there exists a totally invariant fixed point. Taking such a point to infinity turns the map
into a polynomial in the ordinary sense. In $\Rat_{d}$, or generally in $\mathbb{P}^{N} = \mathbb{P}^{2d+1}$, a map is polynomial iff it is in the closure of the
$\PGL(2)$-orbit of the subvariety defined by zeros in all coefficients in the denominator except the $y^{d}$-coefficient.\end{defn}

\begin{rem}A totally invariant fixed point is not necessarily a totally fixed point. A totally invariant fixed point is one that is totally ramified. A totally
fixed point is the root of the fixed point polynomial when it is unique, i.e. when the polynomial is a power of a linear term. In fact by an easy computation, a map
has a totally invariant, totally fixed point $x$ iff it is degenerate linear with a multiplicity-$d-1$ bad point at $x$, in which case it is necessarily
unstable.\end{rem}

\medskip The polynomial maps define a curve in $\mathrm{M}_{2}^{ss}$; we will show,

\begin{prop}\label{Example}The polynomial curve in $\mathrm{M}_{2}^{ss}$ only satisfies semistable reduction with nontrivial bundles.\end{prop}

\begin{proof}First, note that in $\mathbb{P}^{5}$, the polynomial maps are those that can be conjugated to the form $(a_{0}x^{2} + a_{1}xy +
a_{2}y^{2})/b_{2}y^{2}$, in which case the totally invariant fixed point is $\infty = (1:0)$. We will call the polynomial map locus $X$. If $a_{0} = 0$ then the map
is unstable; we will show that every curve in $X$ contains a map for which $a_{0} = 0$. Clearly, the set of all maps with a given totally invariant fixed point is
isomorphic to $\mathbb{P}^{3}$, and the unstable locus within it is isomorphic to $\mathbb{P}^{2}$ as a linear subvariety, so for there to be any hope of a trivial
bundle, a curve in $X$ cannot lie entirely over one totally invariant point.

\medskip Now, the fixed point equation for a map of the form $f/g$ is $fy - gx$; the homogeneous roots of this equation are the fixed points, with the correct
multiplicities. For our purposes, when the totally invariant point is $\infty$, the fixed point equation is $a_{0}x^{2}y + (a_{1} - b_{2})xy^{2} + a_{2}y^{3}$. We
get that $a_{0} = 0$ iff the totally invariant point is a repeated root of the fixed point equation.

\medskip There exists a map from $X$ to $\mathbb{P}^{1} \times \mathbb{P}^{2}$, mapping $\varphi$ to its totally invariant point in $\mathbb{P}^{1}$, and to the two
elementary symmetric polynomials in the two other fixed points in $\mathbb{P}^{2}$. Write $(x:y)$ for the image in $\mathbb{P}^{1}$ and $(a:b:c)$ for the image in
$\mathbb{P}^{2}$. Now $(x:y)$ is a repeated root if $ax^{2} + bxy + cy^{2} = 0$. The equation defines an ample divisor, so every curve in $\mathbb{P}^{1} \times
\mathbb{P}^{2}$ will meet it. Finally, a curve in $X$ maps either to a single point in $\mathbb{P}^{1} \times \mathbb{P}^{2}$, in which case it must contain points
with $a_{0} = 0$ as above, or to a curve, in which case it intersects the divisor $ax^{2} + bxy + cy^{2} = 0$. In both cases, the curve contains unstable points.
Thus there is no global semistable curve $D$ in $\Hom_{d}^{n, ss}$ mapping down to $C$.\end{proof}

\medskip Note that in the above proof, maps conjugate to $x^{2}$ have two totally invariant points, so a priori the map from $X$ to $\mathbb{P}^{1} \times
\mathbb{P}^{2}$ is not well-defined at them. However, for any curve $D$ in $X$, there is a well-defined completion of this map, whose value at $x^{2}$ on the
$\mathbb{P}^{1}$ factor is one of the two totally invariant points. Thus this complication does not invalidate the above proof.

\medskip Let us now compute the vector bundle classes that do occur for the polynomial curve. We work with the description $x^{2} + c$, which yields an affine curve
that maps one-to-one into $C$, missing only the point at infinity, which is conjugate to $\frac{x^{2} - x}{0}$. To hit the point at infinity, we choose the
alternative parametrization $cx^{2} - cx + 1$, which, when $c = \infty$, corresponds to the unique (up to conjugation) semistable degenerate constant map. For any
$c$, this map is conjugate to $x^{2} - cx + c$ and thence $x^{2} + c/2 - c^{2}/4$, using the transition function $[c, -1/2; 0, 1]$. Thus the bundle splits as
$\mathcal{O} \oplus \mathcal{O}(1)$.

\medskip This bundle depends on the choice of $D$. In fact, if we choose another parametrization for $D$, for example $c^{2}x^{2} - c^{2}x + 1$, then the transition
function $[c^{2}, -1/2; 0, 1]$, which leads to the bundle $\mathcal{O} \oplus \mathcal{O}(2)$. This is not equivalent to $\mathcal{O} \oplus \mathcal{O}(1)$. This
then leads to the question of which classes of bundles can occur over each $C$. In the example we have just done, the answer is every nontrivial class: for every
positive integer $m$, we can use $c^{m}x^{2} - c^{m}x + 1$ as a parametrization, leading to $\mathcal{O} \oplus \mathcal{O}(m)$, which exhausts all nontrivial
projective bundle classes.

\medskip Recall the result of Proposition~\ref{multi}:

\begin{prop}Suppose $C$ is isomorphic to $\mathbb{P}^{1}$, and there exists $U \subseteq \Hom_{d}^{n, ss}$ mapping finite-to-one into $C$ such that $U$ is a
projective curve minus a point. Then there are always infinitely many possible classes: if the class of $U$ is thought of as splitting as $\mathbf{P}(\mathcal{E}) =
\mathcal{O} \oplus \mathcal{O}(m_{1}) \oplus \ldots \oplus \mathcal{O}(m_{n})$, where $m_{i} \in \mathbb{N}$, then for every integer $l$ the class $\mathcal{O}
\oplus \mathcal{O}(lm_{1}) \oplus \ldots \oplus \mathcal{O}(lm_{n})$ also occurs.\end{prop}

\begin{proof}Imitating the analysis of the polynomial curve above, we can parametrize $C$ by one variable, say $c$, and choose coordinates such that the sole bad
point in the closure of $U$ corresponds to $c = \infty$. Now, we can by assumption find a piece $U'$ above the infinite point with a transition function determining
the vector bundle $\mathcal{O} \oplus \mathcal{O}(m_{1}) \oplus \ldots \oplus \mathcal{O}(m_{n})$. Now let $V$ be the composition of $U'$ with the map $c \mapsto
c^{l}$. Then $U$ and $V$ determine a vector bundle satisfying semistable reduction, of class $\mathcal{O} \oplus \mathcal{O}(lm_{1}) \oplus \ldots \oplus
\mathcal{O}(lm_{n})$, as required.\end{proof}

\medskip The example in Theorem~\ref{Example}, of polynomial maps, is equivalent to a multiplier condition. When $d = 2$, a map is polynomial iff it has a
superattracting fixed point, i.e. one whose multiplier is zero; see the description in the first chapter of \cite{ADS}. One can imitate the  proof that semistable
reduction does not hold for a more general curve, defined by the condition that there exists a fixed point of multiplier $t \neq 1$. In that case, the condition
$b_{1} = 0$ is replaced by $b_{1} = ta_{0}$, and the point is a repeated root of the fixed point equation iff $a_{0} = b_{1}$, in which case we clearly have $a_{0}
= b_{1} = 0$ and the point is unstable.

\medskip When the multiplier is $1$, the fixed point in question is automatically a repeated root, with $b_{1} = a_{0}$. The condition that the point is the only
fixed point corresponds to $b_{2} = a_{1}$, which by itself does not imply that the map fails to be a morphism, let alone that it is unstable.

\medskip Instead, the condition that gives us $b_{1} = a_{0} = 0$ is the condition that the fixed point is totally invariant. Specifically, the fixed point's two
preimages are itself and one more point; when the fixed point is $\infty$, the extra point is $-b_{2}/b_{1}$. Now we can map $X$ to $\mathbb{P}^{1} \times
\mathbb{P}^{1}$ where the first coordinate is the fixed point and the second is its preimage. This map is well-defined on all of $X$ because only one point can be a
double root of a cubic. Now the diagonal is ample in $\mathbb{P}^{1} \times \mathbb{P}^{1}$, so the only way a curve $D$ can avoid it is by mapping to a single
point; but in that case, $D$ lies in a fixed variety isomorphic to $\mathbb{P}^{3}$ where the unstable locus is $\mathbb{P}^{2}$, so it will intersect the unstable
locus.

\medskip The fact that any condition of the form ``there exists a fixed point of multiplier $t$'' induces a curve for which semistable reduction requires a
nontrivial bundle means that there is no hope of enlarging the semistable space in a way that ensures we always have a trivial bundle. We really do need to think of
semistable reduction as encompassing nontrivial bundle classes as well as trivial ones.

\medskip Specifically: it is trivial to show that the closure of the polynomial locus in $\Rat_{2}$ includes all the unstable points (fix $\infty$ to be the totally
invariant point and let $a_{0}$ go to zero). At least some of those unstable points will also arise as closures of other multiplier-$t$ conditions. However,
different multiplier-$t$ conditions limit to different points in $\mathrm{M}_{2}^{ss}\setminus\mathrm{M}_{2}$.

\section{The General Case}\label{pfbad}

\noindent So far we have talked about nontrivial classes in $\mathrm{M}_{2}$. But we have a stronger result, restating Theorem~\ref{bad1}:

\begin{thm}\label{bad2}For all $n$ and $d$, over any base field, there exists a curve with no trivial bundle class satisfying semistable reduction.\end{thm}

\begin{proof}In all cases, we will focus on \textbf{polynomial maps}, which we will define to be maps that are $\PGL(n+1)$-conjugate to maps for which the last
polynomial $q_{n}$ has zero coefficients in every monomial except possibly $x_{n}^{d}$.

\begin{lem}\label{poly1}The set of polynomial maps, defined above, is closed in $\overline{\Hom_{d}^{n}} = \mathbb{P}^{N}$.\end{lem}

\begin{proof}Clearly, the set of polynomial maps with respect to a particular hyperplane -- for example, $x_{n} = 0$ -- is closed. Now, for each hyperplane
$a_{0}x_{0} + \ldots + a_{n}x_{n} = 0$, we can check by conjugation to see that the condition that the map is polynomial corresponds to the condition that
$a_{0}q_{0} + \ldots + a_{n}q_{n} = c(a_{0}x_{0} + \ldots + a_{n}x_{n})^{d}$, where $c$ may be zero. As $\mathbb{P}^{n}$ is proper, it suffices to show that the
condition ``$\varphi$ is polynomial with respect to $a_{0}x_{0} + \ldots + a_{n}x_{n} = 0$'' is closed in $\left(\mathbb{P}^{n}\right)^{*} \times \mathbb{P}^{N}$.

\medskip Now, we may construct a rational function $f$ from $\left(\mathbb{P}^{n}\right)^{*} \times \mathbb{P}^{N}$ to $\Sym^{d}(\mathbb{P}^{n}) \times
\Sym^{d}(\mathbb{P}^{n})$ by $((a_{0}x_{0} + \ldots + a_{n}x_{n}), \varphi) \mapsto ((a_{0}x_{0} + \ldots + a_{n}x_{n})^{d}, a_{0}q_{0} + \ldots + a_{n}q_{n})$. The
map $\varphi$ is polynomial with respect to $a_{0}x_{0} + \ldots + a_{n}x_{n} = 0$ iff $f$ is ill-defined at $((a_{0}x_{0} + \ldots + a_{n}x_{n}), \varphi)$ or
$f((a_{0}x_{0} + \ldots + a_{n}x_{n}), \varphi) \in \Delta$, the diagonal subvariety. The ill-defined locus of $f$ is closed, and the preimage of $\Delta$ is closed
in the well-defined locus.\end{proof}

\medskip In fact, the condition of $\varphi$ being polynomial with respect to any number of distinct hyperplanes in general position -- in other words, the
condition that $\varphi$ is conjugate to a map for which $q_{i} = c_{i}x_{i}^{d}$ for all $i > 0$ (or $i > 1$, etc.) -- is more or less closed as well. It is not
closed, but a sufficiently good condition is closed. Namely:

\begin{lem}\label{flag}For each $1 \leq i \leq n$, consider the $\PGL(n+1)$-orbit of the space of maps in which, for each $j \geq i$, $q_{j}$ has zero coefficients
in every monomial containing any term $x_{k}$ with $k < j$. This orbit is closed in $\mathbb{P}^{N}$.\end{lem}

\begin{proof}Observe that the above-defined space of maps consists of maps that are polynomial with respect to $x_{n} = 0$, such that the induced map on the totally
invariant hyperplane $x_{n} = 0$ is polynomial with respect to $x_{n-1} = 0$, and so on until we reach the induced map on the totally invariant subspace $x_{i+1} =
\ldots = x_{n} = 0$.

\medskip Now we use descending induction. Lemma~\ref{poly1} is the base case, when $i = n$. Now suppose it is true down to $i$. Then for $i - 1$, the condition of
having no nonzero $x_{k}$ term in $q_{i-1}$ with $k < i - 1$ is equivalent to the condition that the induced map on the totally invariant subspace $x_{i} = x_{i+1}
= \ldots = x_{n} = 0$ is polynomial; this condition is closed in the space of all maps that are polynomial down to $x_{i}$, which we assume closed by the induction
hypothesis.\end{proof}

\begin{defn}We call maps of the form in the above lemma \textbf{polynomial with respect to $B$}, where $B$ is the Borel subgroup preserving the ordered basis of
conditions. In the case above, $B$ is the upper triangular matrices.\end{defn}

\medskip We need one final result to make computations easier:

\begin{lem}\label{last}Let $X$ be a curve of polynomial maps, all with respect to a Borel subgroup $B$, and let $\varphi$ be a semistable map in
$\overline{\PGL(n+1)\cdot X}$. Then $\varphi \in \overline{B\cdot X}$.\end{lem}

\begin{proof}Let $C$ be the closure of the image of $X$ in $\mathrm{M}_{d}^{n, ss}$. By semistable reduction, there exists some affine curve $Y \ni \varphi$ mapping
finite-to-one to $C$, i.e. dominantly. We need to find some open $Z \subseteq Y$ containing $\varphi$ and some $f: Z \to \PGL(n+1)$ such that $f(\varphi)$ is the
identity matrix, and $Z' = \{(f(z)\cdot z)\}$ consists of maps which are polynomial with respect to $B$. Such map necessarily exists: we have a map $h$ from $Y$ to
the flag variety of $\mathbb{P}^{n}$ sending each $y$ to the subgroup with respect to which it is polynomial (possibly involving some choice if generically $y$ is
polynomial with respect to more than one flag), which then lifts to $G$, possibly after deleting finitely many points. Generically, a point of $X$ maps to a point
of $C$ that is in the image of $Z$; therefore, picking the correct points in $X$, we get that $\varphi \in \overline{B\cdot X}$.\end{proof}

\medskip With the above lemmas, let us now prove the theorem with $n = 1$, which is slightly easier than the higher-$n$ case, where the more complicated
Lemma~\ref{flag}. We will use the family $x^{d} + c$, where $c \in \mathbb{A}^{1}$. In projective notation, this is $\frac{a_{0}x^{d} + a_{d}y^{d}}{b_{d}y^{d}}$,
which is a one-dimensional family modulo conjugation. We have,

\begin{lem}Let $V$ be the closure of the $\PGL(2)$-orbit of the family $\frac{a_{0}x^{d} + a_{d}y^{d}}{b_{d}y^{d}}$ in $\mathbb{P}^{N}$. Then:

\begin{enumerate}
  \item In characteristic $0$ or $p \nmid d$, every $\varphi \in V$ is actually in the $\PGL(2)$-orbit of the family, or else it is a degenerate linear map, conjugate
  to $\frac{a_{d-1}xy^{d-1} + a_{d}y^{d}}{b_{d}y^{d}}$.
  \item In characteristic $p \mid d$, with $p^{m} \mid\mid d$ and $p^{m} \neq d$, every $\varphi \in V$ is in the $\PGL(2)$-orbit of the family or is a degenerate map
  conjugate to $\frac{a_{d-p^{m}}xy^{d-p^{m}} + a_{d}y^{d}}{b_{d}y^{d}}$.
  \item In characteristic $p$ with $d = p^{m}$, set $V$ to be the closure of the orbit of the family $\frac{a_{0}x^{d} + a_{d-1}xy^{d-1}}{b_{d}y^{d}}$; then every
  $\varphi \in V$ is actually in the orbit of the family, or else it is a degenerate linear map, conjugate to $\frac{a_{d-1}xy^{d-1} + a_{d}y^{d}}{b_{d}y^{d}}$, and
  furthermore $a_{d-1} = b_{d}$.
\end{enumerate}
\end{lem}

\begin{proof}Observe that the first two cases are really the same: case $2$ is reduced to case $1$ viewed as a degree-$\frac{d}{p^{m}}$ map in
$(x^{p^{m}}:y^{p^{m}})$. So it suffices to prove case $1$ to prove $2$; we will start with the family $\frac{a_{0}x^{d} + a_{d}y^{d}}{b_{d}y^{d}}$ and see what
algebraic equations its orbit satisfies. As polynomials are closed in $\overline{\Rat_{d}}$, every point in the closure of the orbit is a polynomial. We may further
assume it is polynomial with respect to $y = 0$; therefore, by Lemma~\ref{last}, it suffices to look at the action of upper triangular matrices. Further, the
condition of being within the family $\frac{a_{0}x^{d} + a_{d}y^{d}}{b_{d}y^{d}}$ is stabilized by diagonal matrices; therefore, it suffices to look at the action
of matrices of the form $[1, t; 0, 1]$.

\medskip Now, the conjugation action of $[1, t; 0, 1]$ fixes $b_{d}y^{d}$ and maps $a_{0}x^{d} + a_{d}y^{d}$ to $a_{0}(x - ty)^{d} + (a_{d} + tb_{d})y^{d}$. Clearly,
there is no hope of obtaining any condition on $b_{d}$ or $a_{d}$. Now, the conditions on the terms $a_{0}, \ldots, a_{d-1}$ are that for some $t$, they fit into
the pattern $a_{0}(x^{d} - dtx^{d-1}y + \ldots \pm dt^{d-1}xy^{d-1})$, i.e. $a_{i} = (-t)^{i}{d \choose i}a_{0}$. To remove the dependence on $t$, note that when $i
+ j = k + l$, we have ${d \choose i}{d \choose j}a_{i}a_{j} = {d \choose k}{d \choose l}a_{k}a_{l}$, as long as $i, j, k, l < d$.

\medskip Let us now look at what those conditions imply. Setting $j = i, k = i-1, l = i+1$, we get conditions of the form ${d \choose i}^{2}a_{i}^{2} = {d \choose
i-1}{d \choose i+1}a_{i-1}a_{i+1}$, whenever $i + 1 < d$. If $a_{0} \neq 0$, then the value of $a_{1}$ uniquely determines the value of $a_{2}$ by the condition
with $i = 1$; the value of $a_{2}$ uniquely determines $a_{3}$ by the condition with $i = 2$; and so on, until we uniquely determine $a_{d-1}$. In this case,
choosing $t = -\frac{a_{1}}{da_{0}}$ will conjugate this map back to the family $\frac{a_{0}x^{d} + a_{d}y^{d}}{b_{d}y^{d}}$. If $a_{0} = 0$, then the equation with
$i = 1$ will imply that $a_{1} = 0$; then the equation with $i = 2$ will imply that $a_{2} = 0$; and so on, until we set $a_{d-2} = 0$. We cannot ensure $a_{d-1} =
0$ because $a_{d-1}$ always appears in those equations multiplied by a different $a_{i}$, instead of squared. Hence we could get a degenerate-linear map.

\medskip In case $3$, we again look at the action of matrices of the form $[1, t; 0, 1]$. Such matrices map $\frac{a_{0}x^{d} + a_{d-1}xy^{d-1}}{b_{d}y^{d}}$ to
$\frac{a_{0}x^{d} + a_{d-1}xy^{d-1} + (-a_{0}t^{d} - a_{d-1}t + b_{d}t)y^{d}}{b_{d}y^{d}}$. Now the only way a map of the form $\frac{a_{0}x^{d} + a_{d-1}xy^{d-1} +
a_{d}y^{d}}{b_{d}y^{d}}$ could degenerate is if the image of the polynomial map $t \mapsto -a_{0}t^{d} - a_{d-1}t + b_{d}t$ misses $a_{d}$, which could only happen
if the polynomial were constant, i.e. $a_{0} = 0$ and $a_{d-1} = b_{d}$, giving us a degenerate-linear map.\end{proof}

\begin{rem}The importance of the lemma is that in all degenerate cases, the map is necessarily unstable, since $d-1$ (or, in case $2$, $d-p^{m}$) is always at least
as large as $d/2$.\end{rem}

\medskip We can now prove the theorem for $n = 1$. So if we can always find a $D \subseteq \Hom_{d}^{n, ss}$ that works globally, we can find one over a family in
which every map is conjugate to $\frac{a_{0}x^{d} + a_{d}y^{d}}{b_{d}y^{d}}$, or, in characteristic $p$ with $d = p^{m}$, $\frac{a_{0}x^{d} +
a_{d-1}xy^{d-1}}{b_{d}y^{d}}$. It suffices to show that there exists a map with $a_{0} = 0$. For this, we use the fixed point polynomial, which is well-defined on
this family. If the polynomial is fixed, then all maps in the family may be simultaneously conjugated to the form $\frac{a_{0}x^{d} + a_{d}y^{d}}{b_{d}y^{d}}$ (or
$\frac{a_{0}x^{d} + a_{d-1}xy^{d-1}}{b_{d}y^{d}}$), and then one map must have $a_{0} = 0$. If the polynomial varies, then some map will have the point at infinity
colliding with another fixed point. This will force the map to be ill-defined at infinity; recall that totally invariant points are simple roots of the fixed point
polynomial, unless they are bad. This will force $a_{0}$ to be zero, again.

\medskip For higher $n$, the proof is similar. The lemma we need is similar to the lemma we use above, but is somewhat more complicated:

\begin{lem}Let $V$ be the closure of the $\PGL(n+1)$-orbit of the family $(c_{0}x_{0}^{d} + bx_{1}^{d}:q_{1}:\ldots:q_{n})$, where $q_{i}$ is $x_{j}$-free for all
$j < i$.

\begin{enumerate}
  \item If the characteristic does not divide $d$, then every $\varphi \in V$ is actually in the $\PGL(n+1)$-orbit of the family, or else it is a degenerate map,
  whose only possible nonzero coefficients in $q_{0}$ are those without an $x_{0}$ term and those of the form $x_{0}p_{0}$ where there is no nonzero $x_{0}$-term in
  $p_{0}$.
  \item If the characteristic $p$ satisfies $p \mid d$, with $d \neq p^{m} \mid\mid d$ then the same statement as in case $1$ holds as long as each $q_{i}$ is in
  terms of $x_{j}^{p^{m}}$, but with $x_{0}p_{0}$ replaced by $x_{0}^{p^{m}}p_{0}$.
  \item If the characteristic $p$ satisfies $d = p^{m}$ then, changing the family to $(c_{0}x_{0}^{d} + bx_{0}x_{1}^{d-1}:q_{1}:\ldots:q_{n})$, with $q_{i}$ in
  terms of $x_{j}^{d}$ as in case $2$, the same statement as in case $1$ holds.
\end{enumerate}
\end{lem}

\begin{proof}As in the one-dimensional case, case $2$ is reducible to case $1$ with $d$ replaced with $\frac{d}{p^{m}}$ and $x_{i}$ with $x_{i}^{m}$. By
Lemma~\ref{last}, we only need to conjugate by upper triangular matrices. Further, we only need to conjugate by just matrices of the family $E$, with first row $(1,
t_{1}, \ldots, t_{n})$ and other rows the same as the identity matrix. This is because we can control the diagonal elements because the condition of being in the
family is diagonal matrix-invariant, and we can control the rest by projecting any curve $Z$ of unipotent upper triangular matrices onto $E$.

\medskip Set $a_{\mathbf{d}}$ to be the $\mathbf{x^{d}}$-coefficient in $q_{0}$. For all vectors $\mathbf{i}$, $\mathbf{j}$, $\mathbf{k}$, $\mathbf{l}$ with
$\mathbf{i} + \mathbf{j} = \mathbf{k} + \mathbf{l}$, we have ${d \choose \mathbf{i}}{d \choose \mathbf{j}}a_{\mathbf{i}}a_{\mathbf{j}} = {d \choose \mathbf{k}}{d
\choose \mathbf{l}}a_{\mathbf{k}}a_{\mathbf{l}}$, as long as none of $\mathbf{i}$, $\mathbf{j}$, $\mathbf{k}$, or $\mathbf{l}$ is in the span of $\mathbf{e}_{i}$
for $i > 0$. Note that $i$ and $\mathbf{i}$ are two separate quantities, one an index of coordinates and one an index of monomials.

\medskip As in the one-dimensional case, we may set $\mathbf{j} = \mathbf{i}$ and $\mathbf{k} = \mathbf{i} - \mathbf{e}_{0} + \mathbf{e}_{i}$. If $c_{0} = a_{(d, 0,
\ldots, 0)} \neq 0$, then by the same argument as before, the values of the $x_{0}^{d-1}x_{i}$-coefficients determine all the rest, and we can conjugate the map
back to the desired form. And if $c_{0} = 0$, then the value of every coefficient that can occur as $\mathbf{i}$ in the above construct is zero; the only
coefficients that cannot are those with no $x_{0}$ component and those with a linear $x_{0}$ component.

\medskip In case $3$, we restrict to matrices of the same form as in case $1$, and observe that those matrices only generate extra $x_{i}^{d}$ and
$x_{i}x_{1}^{d-1}$ in $q_{0}$. The statement is vacuous if $c_{0} = 0$, so assume $c_{0} \neq 0$. For $i = 1$, this is identical to the one-dimensional case, so if
$c_{0} \neq 0$ then we can find an appropriate $t_{1}$. For higher $i$, if $b \neq 0$ then we can extract $t_{i}$ from the $x_{i}x_{1}^{d-1}$ coefficient, which
will necessarily work for the $x_{i}^{d}$ coefficient as well, making the map conjugate to the family; if $b = 0$, then the same equations as for $i = 1$ hold for
higher $i$, and we can again find $t_{i}$'s conjugating the map to the family.\end{proof}

\medskip While we could also control the terms involving a linear (or $p$-power) $x_{0}$ coefficient in the above construction, it is not necessary for our purposes.

\medskip To finish the proof of the theorem, first note that in the closure of the family above, any map for which $c_{0} = 0$ is unstable. Indeed, the
one-parameter subgroup of $\PGL(n+1)$ with diagonal coefficients $t_{0} = n, t_{i} = -1$ for $i > 0$, shows instability. Recall that a map is unstable with respect
to such a family if $t_{i} > t_{0}d_{0} + \ldots + t_{n}d_{n}$ whenever the $x_{0}^{d_{0}}\ldots x_{n}^{d_{n}}$-coefficient of $q_{i}$ is nonzero. With the above
one-parameter subgroup, we have $t_{0}d_{0} + \ldots + t_{n}d_{n} = -d < -1$ for the only nonzero monomials in $q_{i}$ with $i > 0$; in $q_{0}$, the maximal value
of $t_{0}d_{0} + \ldots + t_{n}d_{n}$ is $t_{0} + t_{i}(d-1) = n - (d-1) < n$.

\medskip Now we need to show only that for some map in the family, $c_{0}$ will indeed be zero. So suppose on the contrary that $c_{0}$ is never zero. Then all maps
are, after conjugation, in the family $(c_{0}x_{0}^{d} + bx_{1}^{d}:q_{1}\ldots:q_{n})$, where the linear subvariety $q_{i} = q_{i+1} = \ldots = q_{n}$ is totally
invariant. Now look at the action on the line $x_{2} = \ldots = x_{n} = 0$. Every morphism will induce a morphism on this line, so there will be three fixed points
on it, counting multiplicity. We now imitate the proof in the one-dimensional case: the totally invariant fixed point on this line, $(1:0:\ldots:0)$, will collide
with another fixed point, so the map will be ill-defined at it. This means that $(1:0:\ldots:0)$ is a bad point, which cannot happen unless $c_{0} = 0$.\end{proof}

\medskip Trivially, the above theorem for curves shows the same for higher-dimensional families in $\mathrm{M}_{d}^{n, ss}$. An interesting question could be to
generalize semistable reduction to higher-dimensional families, for which we may get projective vector bundles just like in the case of curves. Trivially, if we
have two proper subvarieties of $\mathrm{M}_{d}^{n, ss}$, $V_{1} \subseteq V_{2}$, and a bundle class occurs for $V_{2}$, then its restriction to $V_{1}$ occurs for
$V_{1}$. In particular, if we have the trivial class over $V_{2}$ then we also have it over $V_{1}$, as well as any other subvariety of $V_{2}$. This leads to the
following question: if the trivial class occurs for every proper closed subvariety of $V_{2}$, does it necessarily occur for $V_{2}$? What if we weaken the
condition and only require the trivial class to occur for subvarieties that cover $V_{2}$?

\section{The Trivial Bundle Case}\label{goodcase}

\noindent For most curves $C \subseteq \mathrm{M}_{d}^{n, ss}$, there occurs a trivial bundle. Since the complement of $\Hom_{d}^{n, ss}$ in $\mathbb{P}^{N}$ has
high codimension, this is true by simple dimension counting. Therefore, it is useful to analyze those curves separately, as we have more tools to work with.
Specifically, we can use more machinery from geometric invariant theory. We will start by proving Proposition~\ref{GIT}, restated below:

\begin{prop}\label{GIT2}Let $X$ be a projective variety over an algebraically closed field with an action by a geometrically reductive linear algebraic group $G$.
Using the terminology of geometric invariant theory, let $D$ be a complete curve in the stable space $X^{s}$ whose quotient by $G$ is a complete curve $C$; say the
map from $D$ to $C$ has degree $m$. Suppose the stabilizer is generically finite, of size $h$, and either $D$ or $C$ is normal. Then there exists a finite subgroup
$S_{D} \subseteq G$, of order equal to $mh$, such that for all $x \in D$ and $g \in G$, $gx \in D$ iff $g \in S_{D}$.\end{prop}

\begin{proof}For $x \in D$, we define $S_{D}(x) = \{g \in G: gx \in D\}$. This is a map of sets from an open dense subset of $D$ to $\Sym^{mh}(G)$, and is regular
on an open dense subset. We have:

\begin{lem}\label{sym}The map from $\Sym^{mh}(G) \times X^{s}$ to $\Sym^{mh}(X^{s}) \times X^{s}$ defined by sending each $(\{g_{1}, \ldots, g_{mh}\}, x)$ to
$(\{g_{1}\cdot x, \ldots, g_{mh}\cdot x\}, x)$ is proper.\end{lem}

\begin{proof}By standard geometric invariant theory, the map from $G \times X^{s}$ to $X^{s} \times X^{s}$, $(g, x) \mapsto (g\cdot x, x)$, is proper. Therefore the
map from $G^{mh} \times (X^{s})^{mh}$ to $(X^{s})^{mh} \times (X^{s})^{mh}$ defined by $(g_{i}, x_{i}) \mapsto (g_{i}\cdot x_{i}, x_{i})$ is also proper, as the
product of proper maps. Now closed immersions are proper, so the map remains proper if we restrict it to $G^{mh} \times X^{s}$ where we embed $X^{s}$ into
$(X^{s})^{mh}$ diagonally; the image of this map lands in $(X^{s})^{mh} \times X^{s}$. Finally, we quotient out by the symmetric group $S_{k}$, obtaining:

\begin{equation*}
  \xymatrix@R+2em@C+2em{
  G^{mh} \times X^{s} \ar[r] \ar[d]_{\pi} & (X^{s})^{mh} \times X^{s} \ar[d]^{\pi} \\
  \Sym^{mh}(G) \times X^{s} \ar[r] & \Sym^{mh}(X^{s}) \times X^{s}
  }
 \end{equation*}

\medskip The map on the bottom is already separated and finite-type; we will show it is universally closed. Extend it by some arbitrary scheme $Y$. If $V \subseteq
\Sym^{mh}(G) \times X^{s} \times Y$ is closed, then so is $\pi^{-1}(V) \subseteq G^{mh} \times X^{s} \times Y$. The map on top is universally closed, so its image
is closed in $(X^{s})^{mh} \times X^{s} \times Y$. But the map on the right is proper, so the image of $V$ is also closed in $\Sym^{mh}(X^{s}) \times X^{s} \times
Y$.\end{proof}

\medskip Now, the rational map $f_{D}(x) = S_{D}(x)\cdot x \in \Sym^{mh}(D)$ can be extended to a morphism on all of $D$, since both $D$ and $\Sym^{mh}(D)$ are
proper. This is trivial if $D$ is normal; if it is not normal, but $C$ is normal, then observe that the map factors through $C$ since it is constant on orbits, and
then analytically extend it through $C$. But now $(f_{D}(x), x)$ embeds into $\Sym^{mh}(X^{s}) \times X^{s}$ as a proper curve. The preimage in $\Sym^{mh}(G) \times
X^{s}$ of this curve is also proper; for each $(f_{D}(x), x)$, it is a finite set of points of the form $(S, x)$ satisfying $S\cdot x = f_{D}(x)$, including
$(S_{D}(x), x)$. Projecting onto the $\Sym^{mh}(G)$ factor, we still get a proper set, which means it must be a finite set of points, as $\Sym^{mh}(G)$ is affine.
One of these points will be $S_{D}$, which is then necessarily finite.

\medskip Finally, if $g, h \in S_{D}$ and $x \in D$ then $g\cdot h\cdot x \in g \cdot D = D$; therefore $S_{D}$ is a group.\end{proof}

\begin{rem}The proposition essentially says that the cover $D \to C$ is necessarily Galois. The generic stabilizer is necessarily a group $H$, normal in
$S_{D}$.\end{rem}

\begin{cor}With the same notation and conditions as in Proposition~\ref{GIT2}, the map from $D$ to $C$ ramifies precisely at points $x \in D$ such that $\Stab(x)$
intersects $S_{D}$ in a strictly larger group than $H$. Furthermore, the ramification degree is exactly $[\Stab(x)\cap S_{D}:H]$.\end{cor}

\medskip For high $n$ or $d$, the stabilized locus of $\Hom_{d}^{n}$ is of high codimension. Furthermore, most curves in $\Hom_{d}^{n, ss}$ lie in $\Hom_{d}^{n,
s}$. Therefore, generically not only is $H$ trivial, but also there are no points on $D$ with nontrivial stabilizer. Thus for most $C$ and $D$, the map $D \to C$
must be unramified. Thus, when $C$ is rational, generically the degree is $1$.

\medskip It's based on this observation that we conjecture the bounds for the nontrivial bundle case in both directions -- that is, that if we fix $C$ and the
bundle class $\mathbf{P}(\mathcal{E})$, then the degree of the map $\pi: D \to C$ is bounded.

\medskip Using the structure result on $\mathrm{M}_{2}^{ss} = \mathbb{P}^{2}$, we can prove much more:

\begin{prop}If $C$ is a generic line in $\mathrm{M}_{2}^{ss}$, then it requires a nontrivial bundle.\end{prop}

\begin{proof}Generically, $C$ is not the line consisting of the resultant locus, $\mathrm{M}_{2}^{ss}\setminus\mathrm{M}_{2}$. So it intersects this line at exactly
one point. Furthermore, since the resultant $\Res_{2}$ is an $\SL(2)$-invariant section, we have $D.\Res_{2} = m\cdot C.\Res_{2}$; we abuse notation and use
$\Res_{d}^{n}$ to refer to the resultant divisor both upstairs and downstairs. Since the degree of the resultant upstairs is $(n+1)d^{n} = 4$~\cite{Jou}, we obtain
$4\cdot D.\mathcal{O}(1) = m$. In other words, $m \geq 4$.

\medskip However, using Proposition~\ref{GIT2}, we will show $m \leq 2$ generically. The generic stabilizer is trivial, and the stabilized locus is a cuspidal
cubic in $\mathbb{P}^{2}$, on which the stabilizer is isomorphic to $\mathbb{Z}/2\mathbb{Z}$, except at the cusp, where it is $S_{3}$. The generic line $C$ will
intersect this cuspidal curve at three points, none of which is the cusp. Therefore, $h = 1$, and there are at most three points of ramification, with ramification
degree $2$. By Riemann-Hurwitz, the maximum $m$ is $2$, contradicting $m \geq 4$.\end{proof}

\bibliographystyle{amsplain}
\bibliography{semistable_reduction_v2}

\bigskip\noindent \sc{Alon Levy, Department of Mathematics, Columbia University, New York, NY 10027, USA}

\noindent \tt{email: levy@math.columbia.edu}

\end{document}